\documentclass[12pt]{amsart}             
\usepackage{amssymb, amsthm, amsmath, amscd, amsfonts}
\usepackage{latexsym}
\usepackage{verbatim}
\usepackage{xcolor}
\usepackage[all]{xy}

\numberwithin{equation}{section}
\newtheorem{theorem}[equation]{Theorem}
\newtheorem{proposition}[equation]{Proposition}

\newtheorem{lemma}[equation]{Lemma}
\newtheorem{definition}[equation]{Definition}

\newtheorem{remark}[equation]{Remark}

\newcommand{\z}{\mathbb Z}       
\newcommand{\n}{\mathbb N}       
\newcommand{\real}{\mathbb R}
\newcommand{\ds}{\displaystyle}

\newcommand{\bunderline}[1]{\underline{#1\mkern-2mu}}
\newcommand{\boverline}[1]{\overline{#1\mkern-2mu}}

\newcommand{\greatersup}{\; \boverline \succ \, \,}
\newcommand{\greaterinf}{\; \bunderline \succ \, \,}

\begin{document}

\title[Subsystems of transitive subshifts with linear complexity]{Subsystems of transitive subshifts with linear complexity}

\author[Dykstra]{Andrew Dykstra}
\address{Andrew Dykstra\\
Department of Mathematics\\
	Hamilton College \\
	Clinton, NY 13323  \\
	 USA}

\email{adykstra@hamilton.edu}

\author[Ormes]{Nicholas Ormes}
\address{Nic Ormes\\
Department of Mathematics\\
University of Denver\\
2390 S. York St.\\
Denver, CO 80208}
	 
\email{nic.ormes@du.edu}

\author[Pavlov]{Ronnie Pavlov}
\address{Ronnie Pavlov\\
Department of Mathematics\\
University of Denver\\
2390 S. York St.\\
Denver, CO 80208}

\email{rpavlov@du.edu}	 

\urladdr{www.math.du.edu/$\sim$rpavlov/}


\begin{abstract}
We bound the number of distinct minimal subsystems of a given transitive subshift of linear complexity, continuing work of Ormes and Pavlov \cite{OP}. We also bound the number of generic measures such a subshift can support based on its complexity function. Our measure-theoretic bounds generalize those of Boshernitzan \cite{B} and are closely related to those of Cyr and Kra \cite{CK}. \end{abstract}

\thanks{The third author gratefully acknowledges the support of NSF grant DMS-1500685.}
\renewcommand{\subjclassname}{MSC 2010}
\maketitle


\section{introduction} \label{introduction}

In this work, we study symbolically defined dynamical systems called subshifts. A subshift is defined by a finite set $\mathcal A$ (called an alphabet), the (left) shift action $\sigma$ on $\mathcal A^{\mathbb{Z}}$, and a set $X \subset \mathcal A^{\mathbb{Z}}$ of sequences that is closed in the product topology and $\sigma$-invariant. 
For convenience, we will refer to a subshift $(X, \sigma)$ only as $X$ since the dynamics are always understood to come from $\sigma$. (See Section~\ref{defs} for more details.)

Given a subshift $X$, let $c_X(n)$ denote the number of words of length $n$ that appear in $X$, i.e., the complexity function of $X$.  Assuming that $X$ is transitive and $c_X(n)$ grows linearly, we ask:  what is the interplay between $c_X(n)$ and the structure of the sub-dynamical systems of $X$?  We study this question in both the topological and measure-theoretic categories.  

In the topological category, we provide bounds on how large $c_X(n)$ must be in order to accommodate a given number of minimal subsystems in $X$.  If $X$ contains only one minimal subsystem, then, by the Morse-Hedlund Theorem (Theorem \ref{MH}), either $X$ is periodic or $c_X(n) \geq n+1$ for every $n$.  In \cite{OP}, Ormes and Pavlov show that, if $X$ is transitive and not minimal, then \begin{equation} \label{OPbound} \limsup_{n \rightarrow \infty} (c_X(n) - 1.5n) = \infty.\end{equation} Moreover they show that the bound (\ref{OPbound}) is {\em sharp} in the sense that the threshold $1.5n$ cannot be increased by any nondecreasing, unbounded function $g: \mathbb{N} \to \n$; for any such $g$, there is an example of a transitive non-minimal subshift $X$ where \[\limsup_{n \rightarrow \infty} (c_X(n) - (1.5n + g(n))) < \infty. \] 
   
For $X$ containing two or more minimal subsystems, in Section \ref{recurrent} we establish the following.

\begin{theorem} \label{maintheoremintro}
	Let $X$ be a transitive subshift which is the orbit closure of a recurrent point $x$, where $X$ has $j \geq 2$ proper minimal subsystems, exactly $i$ of which are infinite {\rm(}$0 \leq i \leq j${\rm)}.  Then the following bounds hold and are sharp: 
	\begin{enumerate}
		\item $\limsup_{n \to \infty} (c_X(n) - (j+i+1)n) = \infty$, and
		\item $\liminf_{n \to \infty} (c_X(n) - (j+i)n) = \infty$.
	\end{enumerate} 
\end{theorem}

The notion that the bounds in Theorem \ref{maintheoremintro} are {\em sharp} is the same as for (\ref{OPbound}), namely, that for any nondecreasing unbounded $g: \n \rightarrow \n$, there exists $X$ satisfying the hypotheses of the theorem for which the bounds do not hold when $g$ is added to $(j+i+1)n$ or $(j+i)n$ respectively.

In Section \ref{nonrecurrent} we consider the case of a general (not necessarily recurrent) transitive subshift. We establish bounds on the growth rate of $c_X(n)$ in the special cases where $X$ contains one or two minimal subsystems and then prove the following.
\begin{theorem} \label{maintheoremnonrecurrentintro}
	Let $X$ be a transitive subshift where $X$ has $j \geq 3$ minimal subsystems, exactly $i$ of which are infinite {\rm(}$0 \leq i \leq j${\rm)}.  Then
	\[\liminf_{n \to \infty}  (c_X(n)  -  (j+i)n) = \infty.\]  
\end{theorem}
Of course, Theorem \ref{maintheoremnonrecurrentintro} implies that $$\limsup_{n \to \infty}  (c_X(n)  -  (j+i)n) = \infty$$ for transitive $X$ as well, and we show that this bound is sharp in the same sense as above.

Turning our attention to the measure-theoretic category, we consider the well-studied problem of bounding the number of ergodic measures that a given subshift can support. For example, in \cite{B}, Boshernitzan shows that if $X$ is minimal and \[\liminf_{n \to \infty}(c_X(n) - Kn) = -\infty,\] then $X$ can support at most $K - 1$ ergodic measures.  He also shows, again assuming minimality, that $X$ is uniquely ergodic provided that 
\[\limsup_{n \to \infty} \frac{c_X(n)}{n} < 3.\]  
Cyr and Kra, motivated by work of Katok \cite{K} and Veech \cite{V} on interval exchange transformations, extended Boshernitzan's work by considering arbitrary (not necessarily minimal) subshifts and nonatomic generic (not necessarily ergodic) measures \cite{CK}. Note that because they are working in such a general setting, their results cannot establish bounds on the number of atomic measures. Indeed, given any subshift $X$, one can always union $X$ with a fixed point to obtain a new subshift $Y$ with an additional (atomic) ergodic measure, where $c_Y(n) = c_X(n)+1$. 

In this paper we consider transitive systems and obtain bounds on the number of generic measures. In particular, we show the following in Section \ref{measures}.

\begin{theorem} \label{measuresupintro}
	Let $X = \overline{\mathcal O(x)}$ be a transitive subshift where $x$ is recurrent and aperiodic.  If \[\limsup_{n \to \infty} \frac{c_X(n)}{n} < 3,\] then $X$ is uniquely ergodic. 
\end{theorem}

The following result is similar to a result obtained in \cite{CK}; see Section~\ref{measures} for more details.

\begin{theorem} \label{measureinfintro}
	Let $X = \overline{\mathcal O(x)}$ be a transitive subshift where $x$ is not eventually periodic in both directions. If \[\liminf_{n \to \infty}(c_X(n) - gn) = -\infty\] for $g \in \n$, then $X$ has at most $g-1$ generic measures.
\end{theorem}
Note that Theorem \ref{measureinfintro} does not imply Theorem \ref{maintheoremnonrecurrentintro}. 
Indeed, if $X$ contains $g$ minimal subsystems, then there are least $g$ generic measures on $X$. Theorem \ref{measureinfintro} would then imply that 
\[\liminf_{n \to \infty}(c_X(n) - gn) > -\infty,\] whereas Theorem \ref{maintheoremnonrecurrentintro} gives the stronger conclusion
\[\liminf_{n \to \infty}(c_X(n) - gn)  = \infty.\]

\section{preliminaries}

\subsection{Subshifts}\label{defs}

 We recall some basic definitions; for more information, see \cite{LM}. 

A {\em full shift} is a pair $(\mathcal A^\z, \sigma)$ where $\mathcal A$ is a finite alphabet, $\mathcal A^\z$ has the product of the discrete topology on $\mathcal A$, and $\sigma: \mathcal A^\z \rightarrow \mathcal A^\z$ is the left-shift defined by $\sigma(x)_i = x_{i+1}$ for each $x = (x_i)_{i \in \z} \in \mathcal A^\z$.  A {\em subshift} is a pair $(X, \sigma)$ where $X$ is a closed and $\sigma$-invariant subset of some $\mathcal A^\z$.  To conserve notation, we will often refer to the subshift $(X, \sigma)$ as simply $X$.

A subshift $X$ is {\em transitive} if there exists $x \in X$ such that $X = \overline{\mathcal O(x)}$, the closure of the orbit $\mathcal O(x) = \{\sigma^n(x) \; : \; n \in \z\}$.  We call such a point $x \in X$ a {\em transitive point}.  If $X = \overline{\mathcal O(x)}$ for every $x \in X$, then $X$ is {\em minimal}. A transitive subshift $X$ is {\em periodic} if it has a transitive point $x$ which is periodic, meaning that there exists $p \in \z$ such that $\sigma^p(x) = x$.  Note that a transitive subshift $X$ is periodic if and only if $X$ has finite cardinality.

Given a subshift $X$, a {\em word of length $n$ in $X$} is a block of symbols $w = w_1w_2\cdots w_n$ that occurs in some point $x \in X$, i.e., $w=x_i x_{i+1} \cdots x_{i+n-1}$ for some $i \in \z$.  Let $\mathcal L_n(X)$ denote the set of all words of length $n$ occuring in some point in $X$, and $\mathcal L(X) = \bigcup_{n = 1}^\infty \mathcal L_n(X)$.  The {\em complexity function} of $X$ is the function $c_X(n): \n \rightarrow \n$ that gives the cardinality of $\mathcal L_n(X)$.  If $X$ is transitive, then $X = \overline{\mathcal O(x)}$ for some $x \in X$, and in this case $c_X(n)$ is equal to the number of words of length $n$ in $x$.

For a symbol $a\in \mathcal{A}$ and $n \geq 1$, the expression $a^n$ denotes the word of length $n$ formed by concatenating $a$ with itself $n$ times. Correspondingly, $a^{\infty}$ denotes the infinite concatenation of $a$ with itself. Depending on the situation, $a^{\infty}$ may denote a bi-infinite sequence, a left-infinite sequence, or a right-infinite sequence. The choice of meaning should be clear from context. 

A point $x \in X$ is {\em recurrent} if every word in $x$ occurs at least twice (equivalently, infinitely often).  If every word in $x$ occurs infinitely often with uniformly bounded gaps between occurrences, then $x$ is {\em uniformly recurrent}.  Note that $x$ is uniformly recurrent if and only if $X = \overline{\mathcal O(x)}$ is minimal.



If $X$ and $Y$ are subshifts, then for any continuous $f: X \rightarrow Y$ such that $f \circ \sigma = \sigma \circ f$, there exist $m, a \in \n$ such that for every $x \in X$, $f(x)_0$ is determined by the word $x_{-m} \cdots x_0 \cdots x_a$.  In this case $f$ is called an $(m+a+1)$-{\em block map} with {\em memory} $m$ and {\em anticipation} $a$.  If such a map $f$ is surjective then it is a {\em factor map}. In this paper, when we define a factor map $f$ on a subshift $X$, it is understood to have codomain $f(X)$. In addition, any $(m+a+1)$-block map has an obvious associated action on finite words as well (for any $n$-letter word $w$, $f(n)$ has length $n - m - a$); we use $f$ to refer to this function also since usage should always be clear from context.

A word $w$ in $\mathcal{L}(X)$ is {\em right-special} if there exist $a, b \in \mathcal{A}$ such that $wa, wb \in \mathcal{L}(X)$ with $a\neq b$. Similarly, $w$ is {\em left-special} if there exist $a, b \in \mathcal{A}$ such that $aw, bw \in \mathcal{L}(X)$ with $a\neq b$. Let $RS_X(n)$ (or just $RS(n)$ if $X$ is understood) denote the set of right-special words in $X$. 

Let $x$ be an element of a subshift $X$. By the {\em omega-limit set of $x$}, we mean the set 
\[\omega(x) = \bigcap_{N \geq 1} \overline{\{ \sigma^n(x) : n \geq N \} }\]
For any $x \in X$, the set $\omega(x)$ is a closed and shift-invariant subset of $X$, so is itself a subshift. 

We say that a point $x$ in a subshift $X$ is {\em eventually periodic to the right}
if there exist integers $p>0$ and $N>0$ such that for all $i>N$, $x_{i} = x_{i+p}$. Similarly, we say that $x$ is  {\em eventually periodic to the left}
if there exist integers $p>0$ and $N>0$ such that for all $i<-N$, $x_{i} = x_{i-p}$.

\subsection{Sturmian subshifts}
There are several different approaches to defining Sturmian subshifts (see \cite{Fogg} for an introduction). We outline one such approach here. 

For any irrational $\beta$, define the map $R_\beta:[0, 1) \rightarrow [0, 1)$ by $R_\beta(x) = x + \beta \mod 1$. For any $x\in (0,1)$, define the sequence $s(x) \in \{0,1\}^{\mathbb{Z}}$ by 
\[s_n(x) = 
\begin{cases}
1 & \text{ if } R_{\beta}^n(x) \in [0,\beta) \\
0 & \text{ if } R_{\beta}^n(x) \in [\beta, 1 ). \\
\end{cases}\]
The bi-infinite sequence $s(x)$ is called a {\em Sturmian sequence} for $\beta$.
For any two-element set $\{a,b\}$, we call a subshift $X \subset \{a,b\}^{\mathbb{Z}}$ a {\em Sturmian subshift} if $X$ can be obtained as the orbit closure of a Sturmian sequence with $0$ replaced by $a$ and $1$ by $b$.  


\subsection{Bounds on the Complexity Function}\label{notation}

Here we introduce shorthand notation for bounds on the complexity function.

\begin{definition} \label{supbound}
Given a function $f: \n \rightarrow \real$, write $c_X(n) \greatersup f(n)$ if $$\limsup_{n \to \infty} (c_X(n) - f(n)) = \infty.$$
\end{definition}

\begin{definition} \label{infbound}
Given a function $f: \n \rightarrow \real$, write $c_X(n) \greaterinf f(n)$ if $$\liminf_{n \to \infty} (c_X(n) - f(n)) = \infty.$$ 
\end{definition}

\noindent Note that the bounds in Theorem \ref{maintheoremintro} can be rewritten using this notation as \begin{enumerate}
\item $c_X(n) \greatersup (j + i + 1)n$ and
\item $c_X(n) \greaterinf (j+i)n$,
\end{enumerate}
and the bound in Theorem \ref{maintheoremnonrecurrentintro} can be rewritten as $c_X(n) \greaterinf (j+i)n$.

As mentioned in the introduction, throughout this work we will say that a bound of the form $c_X(n) \greatersup f(n)$ (or $c_X(n) \greaterinf f(n)$) which holds under some hypotheses on $X$ is \textbf{sharp} if it fails when any nondecreasing unbounded $g: \mathbb{N} \rightarrow \n$ is added to $f(n)$, i.e., if for any such $g$ there exists a subshift $X$ satisfying the relevant hypotheses for which $c_X(n) \greatersup f(n) + g(n)$ (or $c_X(n) \greaterinf f(n) + g(n)$) is false.


\section{Transitive systems with a recurrent transitive point} \label{recurrent}

The main goal of this section is to prove Theorem \ref{maintheoremintro}, which assumes that $X$ contains two or more minimal subsystems.  But first we recall some existing results that can be used to treat $X$ containing a single minimal subsystem.  

\subsection{Single minimal subsystem}

If $X$ is itself minimal and not a periodic orbit, then the following two results establish that 
$$c_X(n) \geq n+1$$
for all $n \geq 1$ and that this bound cannot be improved. 
\begin{theorem}[\cite{MorseHed}] \label{SturmianTheorem}
	If $X$ is a Sturmian subshift, then 
	\[c_X(n) = n+1 \text{ for all } n \geq 1.\]	
\end{theorem}

\begin{theorem}[\cite{MorseHed}] \label{MH}
	Let $X$ be any subshift. If there exists $n \geq 1$ such that $c_X(n) \leq n$, then $X$ is a finite set of periodic points. 
\end{theorem}
Thus the Sturmian subshifts are examples of the lowest complexity subshifts that are not periodic (see \cite{P} for more). If $X$ is transitive and not minimal, then the following (sharp) bound was proved in \cite{OP}. 

\begin{theorem}[\cite{OP}] \label{OrmesPavlovTheorem}
Suppose $X$ is a transitive subshift with a recurrent transitive point $x$.
If $X$ is not minimal, then 
$$c_X(n) \greatersup 1.5n,$$ 
and this bound is sharp.
\end{theorem}

Combining the results of \cite{OP} and techniques of the proof of Theorem \ref{maintheoremintro}, we will be able to establish the following. 

\begin{theorem} \label{Theorem2.5}
Suppose $X$ is a transitive subshift with a recurrent transitive point $x$.
If $X$ properly contains an infinite minimal subsystem, then 
$$c_X(n) \greatersup 2.5n,$$ and this bound is sharp.
\end{theorem}

We postpone the proof of Theorem \ref{Theorem2.5} until after the proof of Theorem  \ref{maintheoremintro}.

\subsection{Multiple minimal subsystems}

For $j \geq 2$ and $0 \leq i \leq j$, consider the set of subshifts $S = S(j,i)$, where $X \in S(j,i)$ if and only if $X$ is transitive with a recurrent transitive point, and $X$ has $j$ distinct minimal subsystems, exactly $i$ of which are infinite. To prove Theorem \ref{maintheoremintro} for $X \in S(j, i)$, we first reduce to the $i = 0$ case via factor maps. 

\subsubsection{Reduction to the $i=0$ case} \label{pimap} 
We define a factor map $\pi$ that maps the $j$ distinct minimal subsystems of $X$ to $j$ distinct points, which are fixed by $\sigma$.
Lemma \ref{factorlemma} will provide an inequality between the complexity sequences for $X$ and $\pi(X)$, which will allow us to simply work with $\pi(X)$ moving forward. 

Let $M_1, \ldots, M_j$ denote the minimal subsystems of $X$, where, without loss of generality, $M_1, \ldots, M_i$ are infinite and $M_{i+1}, \ldots, M_j$ are finite. Since the sets $M_1, \ldots , M_j$ are pairwise disjoint closed subsets of $X$, there is an $r\geq 1$ such that the sets 
$\mathcal{L}_r(M_1),\ldots , \mathcal{L}_r(M_j)$ are pairwise disjoint. Fix such a value of $r$; pick $j$ distinct symbols $a_1, \ldots, a_j$ that do not occur in $x$; and define an $r$-block map $\phi$ with domain $X$, memory zero, and anticipation $(r-1)$ as follows:  \[\phi(y)_q = \begin{cases} a_p & \mbox{ if $y_q \cdots y_{q+r-1} \in \mathcal L_r(M_p)$;} \\ y_q & \mbox{ otherwise.} \end{cases}\] Now pick a symbol $b$ that does not occur in $\phi(x)$, and define a $2$-block map $\psi$ with domain $\phi(X)$, memory zero, and anticipation $1$ as follows:  \[\psi(z)_q = \begin{cases} b & \mbox{ if $z_qz_{q+1} \in \{a_p s, sa_p\}$ for some $p$ and $s \neq a_p$;} \\ y_q & \mbox{ otherwise.} \end{cases}\]   Note that post-composing $\phi$ with $\psi$ has the effect of ensuring that words of the form $a_p^n$ are always preceded and followed by the ``marker" symbol $b$.  Define $\pi = \psi \circ \phi$.  Then the minimal subsystems of $\pi(X)$ are simply the one-point sets $\pi(M_p) = \{a_p^\infty\}$.  If $X \in S(j,i)$, it follows that $\pi(X) \in S(j, 0)$.  

\begin{lemma} \label{factorlemma}
If $X \in S(j,i)$ with $j \geq 2$ and $\pi$ is the factor map defined above, then, for every $n > r$, $c_X(n) \geq c_{\pi(X)}(n-r) + in$.
\end{lemma}

\begin{proof}
Note that for any $q$-block factor map from a subshift $X$ onto a subshift $Y$, a word of length $m$ in $X$ determines a word of length $m-q+1$ in $Y$. It follows that $c_{X}(m) \geq c_Y(m-q+1)$. 
	
Applying this to our situation, since $\psi$ is a $2$-block map, $c_{\phi(X)}(n) \geq c_{\pi(x)}(n-1)$. To complete the proof, it is enough to show that $c_X(n) \geq c_{\phi(X)}(n-r+1)+ in$. For any $n$, let $W_0$ denote the set of words in $\phi(X)$ of the form $a_p^{n-r+1}$ where $n \geq r$ and $1 \leq p \leq i$, and let $W_1 = \mathcal L_{n-r+1}(\phi(X)) \setminus W_0$.  Each word in $W_1$ has at least one $\phi$-preimage in $\mathcal L_n(X)$, but each $a_p^{n-r+1} \in W_0$ has at least $n+1$ preimages in $\mathcal L_n(M_p)$, since any word in $\mathcal L_n(M_p)$ is a preimage of $a_p^{n-r+1}$, and $c_{M_p}(n) \geq n+1$ for all $n$, using Theorem \ref{MH} and the fact that $M_p$ is infinite.
\end{proof}

\begin{lemma} \label{RSlemma}
For each $p \in \{1, \ldots, j\}$ and $n \in \n$, $a_p^n$ is both right- and left-special in $\pi(X)$.
\end{lemma}

\begin{proof}
Since $a_p^\infty \in \pi(X)$, $a_p^n$ can be followed by $a_p$. Assume for a contradiction that $a_p^n$ is not right-special, i.e., it can only be followed by $a_p$. Since $\pi(X) \neq \{a_p^{\infty}\}$, there must exist $m \in \z$ so that \[\sigma^m(x) = \cdots y_{-3}y_{-2}y_{-1}. a_p^\infty,\] where $y_{-1} \neq a_p$.  Since $x$ is recurrent, $y_{-k}\cdots y_{-k+n} = y_{-1}a_p^n$ for some $k \geq n+2$.  Then $-k + n \leq -2$, so there is some smallest $q \geq 1$ such that $y_{-k+n+q} \neq a_p$.  Then $y_{-k+q} \cdots y_{-k+n+q-1} = a_p^n$ and $y_{-k+n+q} \neq a_p$, implying that $a_p^n$ is right-special, a contradiction. Our original assumption was then false, i.e., $a_p^n$ is right-special; a symmetric argument shows that $a_p^n$ is left-special.
\end{proof}

To establish that (1) and (2) from Theorem~\ref{maintheoremintro} hold, first observe that, by Lemma \ref{factorlemma}, it is enough to show only that (1) and (2) hold for $\pi(X)$.  Indeed, if we prove that $c_{\pi(X)}(n) \greatersup (j+1)n$, then there is a strictly increasing sequence $(n_k)$ such that $c_{\pi(X)}(n_k - r) > (j+1)(n_k - r) + k$ for all $k$.  So by Lemma \ref{factorlemma}, \begin{eqnarray*}
	c_X(n_k) & \geq & c_{\pi(X)}(n_k - r ) + in_k \\
	& > & (j+1)(n_k-r ) + k + in_k \\ 
	& = & (j+i + 1)n_k + k - r(j+1),
\end{eqnarray*}
which implies $c_{X}(n) \greatersup (j+i+1)n$.  By a similar argument,  $c_{\pi(X)}(n) \greaterinf jn$ implies $c_{X}(n) \greaterinf (j+i)n$.  

Because of this, to simplify notation we will replace $\pi(X)$ with $X$ and make the following assumptions: \begin{itemize}
	\item {\bf (A1):}  $X \in S(j, 0)$;
	\item {\bf (A2):}  the minimal subsystems of $X$ are the one-point systems $M_1 = \{a_1^\infty\}, \ldots, M_j = \{a_j^\infty\}$; and
	\item {\bf (A3):} for any word in $X$ of the form $a_ps$ or $sa_p$, $s = b$.
\end{itemize}

\subsubsection{Proof that (1) and (2) from Theorem \ref{maintheoremintro} hold when $i=0$} \label{hold}  
Assume $X$ is as above.  Then there is a transitive point $x$ for $X$ that is recurrent, and that therefore cannot be both eventually periodic to the left and eventually periodic to the right. Without loss of generality, assume that $x$ is not eventually periodic to the right. Our proof will involve bounding from below the number of right-special words in $X$ of various lengths (were $x$ eventually periodic to the right, we would instead count left-special words).  The following elementary lemma will be used to verify that $c_X(n) \greatersup (j+1)n$.

\begin{lemma} \label{countinglemma}
For each $n \geq 2$, $c_X(n) \geq c_X(n-1) + \#RS(n-1)$.  Therefore, for $m < n$, $c_X(n) \geq c_X(m) + \sum_{\ell = m}^{n-1} \#RS(\ell)$.
\end{lemma}

\begin{proof}
For every $n$, each word of length $n-1$ can be extended to at least one word of length $n$, while each right-special word of length $n-1$ can be extended to at least two words of length $n$. This yields $c_X(n) \geq c_X(n-1) + \#RS(n-1)$.
Applying this recursively, we obtain the inequality 
$c_X(n) \geq c_X(m) + \sum_{\ell = m}^{n-1} \#RS(\ell)$ for $m<n$. 
\end{proof}

We begin by considering the set of right-special words provided by Lemma \ref{RSlemma}, which we will call $B$: \[B := \{a_p^n \; : \; 1 \leq p \leq j \mbox{ and } n \in \n\}.\] 
By just considering the elements of $B$, we see that 
$\#RS(n) \geq j$ for all $n$. 

By combining the inequality $\#RS(n) \geq j$ with Lemma \ref{countinglemma}, we can show that the bound $c_X(n) \greatersup (j+1)n$ would imply $c_X(n) \greatersup jn$.  Indeed, note that $c_X(n) \greatersup (j+1)n$
implies that there exists an increasing sequence of integers $(n_k)$ such that $c_X(n_k) \geq (j+1)n_k$. Thus for $n > n_k$, 
\begin{align*}
c_X(n) & \geq c_X(n_k) + \sum_{\ell = n_k}^{n-1} \# RS(n) \\
& \geq (j+1)n_k + j(n-n_k) \\
& \geq j n + n_k, 
\end{align*}
which gives $c_X(n) \greaterinf j n$.

Thus we will be done if we establish the bound $c_X(n) \greatersup (j+1)n$. We divide the proof into cases, beginning with the simplest case. 

\

\textbf{Case (i):} For all sufficiently large $n$, $\#RS(n) \geq j+1$ and there is a strictly increasing sequence $(n_k)$ where $\#RS(n_k) \geq j+2$ for all $k$.

The assumptions imply that, for $n>n_k$, 
\[\sum_{\ell=1}^{n} \# RS(\ell) \geq (j+1)n+k.\] This implies that $c_X(n) \greatersup (j+1)n$,
which completes the proof of Theorem \ref{maintheoremintro} in Case (i). 

We now assume that the hypotheses of Case (i) do not hold. That is, precisely one of the following two conditions holds.

\

{\bf Case (ii):} There is a strictly increasing sequence $(n_k)$ such that $\#RS(n_k) = j$, i.e., $RS(n_k) = B$, for all $k$;

{\bf Case (iii):} for all sufficiently large $n$, $\#RS(n) = j+1$.

\begin{lemma} \label{structurelemma}
In Case (ii) or (iii), for each $k$, after re-indexing the minimal subsystems $M_1, \ldots, M_j$ if necessary, there exist words $w_{12}^{(k)}$, $w_{23}^{(k)}$, \ldots, $w_{j1}^{(k)}$, each of which begins and ends with the symbol $b$, such that the transitive point $x \in X$ has one of the following forms:
\[x = \cdots \; a_1^{\geq n_k} \; w_{12}^{(k)} \; a_2^{\geq n_k} \; w_{23}^{(k)} \; \cdots \; a_j^{\geq n_k} \; w_{j1}^{(k)} \; a_1^{\geq n_k} \; w_{12}^{(k)} \; \cdots \ {\rm or}\] 
\[x =  a_1^{\infty} \; w_{12}^{(k)} \; a_2^{\geq n_k} \; w_{23}^{(k)} \; \cdots \; a_j^{\geq n_k} \; w_{j1}^{(k)} \; a_1^{\geq n_k} \; w_{12}^{(k)} \; \cdots,\] 
where each $a_p^{\geq n_k}$ represents a word of the form $a_p^n$ for some $n \geq n_k$.

\end{lemma}

\begin{proof}
Recall that $x$ is a transitive point that is not eventually periodic to the right. For all $p \in \{1,\ldots ,j\}$ and all $n \geq 1$, $a_p^n$ occurs in $x$. Since $x$ is not eventually periodic to the right, $a_p^nb$ must occur as well. Since $x$ is recurrent,
$a_p^nb$ occurs infinitely many times in $x$. It follows that 
$ba_p^n$ also occurs infinitely many times in $x$. 

Assume the hypothesis of Case (ii), and fix $k \geq 1$ and $p \in \{1,\ldots ,j\}$. 
Define $w_1 = b$.  Since we are in Case (ii), $a_p^{n_k-1}w_1 \not \in RS(n_k)$, so there is only one symbol, call it $w_2$, that can appear after $a_p^{n_k-1}w_1$.  Similarly there is only one symbol that can appear after $a_p^{n_k-2}w_1w_2$.  Continuing in this way, each successive symbol $w_i$ is forced (at least) until some $w_i \cdots w_{i + n_k-1} \in RS(n_k)$, i.e., $w_i \cdots w_{i + n_k-1} = a_r^{n_k}$ for some $r$.  There must exist some smallest $i$ for which this is true, since the omega-limit set $\omega(x)$ must contain one of the minimal subsystems $M_1, \ldots, M_j$. Set $w(p) = w_1w_2\cdots w_{i-1}$. Note that $w(p)$ is the only word that begins with $b$ that can follow $a_p^{n_k}$ in $x$, and that $a_r^{n_k}$ is the only word that can follow $w(p)$. Set $f(p) = r$.  

Since $1 \leq p \leq j$ was arbitrary, we obtain a function $f:\{1,\ldots ,j\} \to \{1,\ldots ,j\}$. Since $b a_p^{n_k}$ appears in $x$ for each $p$, 
each $p \in \{1,\ldots ,j\}$ is equal to $f(q)$ for some 
$q \in \{1,\ldots ,j\}$, i.e. $f$ is a bijection and thereby a composition of cyclic permutations. If $f$ were the composition of two or more cyclic permutations, then $x$ would not contain all $a_p^{n_k}$ for $1 \leq p \leq j$, a contradiction to transitivity of $x$. Therefore, $f$ must cyclically permute the elements of $\{1,\ldots ,j\}$. Re-indexing if necessary, we may assume that $f(p)=p+1$ for $p<j$ and $f(j)=1$. For $p \in \{1,\ldots,j\}$, set $w_{pf(p)}^{(k)}=w(p)$. After re-indexing again in the case where $x$ is eventually left-asymptotic, it follows that $x$ has one of the two prescribed forms.  

Now assume the hypothesis of Case (iii). By the assumptions on $x$, there must exist a strictly increasing sequence $(n_k)$ such that $ba_1^{n_k}b$ occurs in $x$ for all $k$. Because $ba_1^{n_k+1}$ also occurs in $x$, we know that $ba_1^{n_k}$ is right-special for every $k$.  By deleting the first few terms of the sequence $(n_k)$ if necessary, this together with the assumption of Case (3) implies that \[RS(n_k+1) = \{ba_1^{n_k}, a_1^{n_k+1}, \ldots, a_j^{n_k+1}\}\] for every $k$.  The word $a_1^{n_k}b \not \in RS(n_k+1)$, so, as in Case (ii), the word  $a_1^{n_k}b$ forces a transition word $w(1)$ (whose first symbol is $b$).  But note that the only symbols that can follow $ba_1^{n_k}$ are also $a_1$ or $b$, which similarly implies that the word $ba_1^{n_k}b$ forces the same transition word $w(1)$. The rest of the argument from Case (ii) carries through with $n_k$ replaced by $n_k+1$. We obtain the same possible forms of $x$ as in Case (ii). 
\end{proof}

We proceed assuming the conclusion of Lemma \ref{structurelemma}.
For each $p \in \{1, \ldots, j\}$ and $k \in \n$, define \[\ell_{k, p} := \min\{\ell \geq n_k \; : \; ba_p^\ell b \in \mathcal L(X)\},\] and set $\ell_{0, p} := 0$.  Consider the following set of words: \[E(k) := \{a_1^{\ell_{k, 1}}w_{12}^{(k)}a_2^{\ell_{k, 2}},\ldots, a_j^{\ell_{k, j}}w_{j1}^{(k)}a_1^{\ell_{k, 1}}\}.\]  It follows from Lemma \ref{structurelemma} that, if $u \in E(k)$, then any suffix of $u$ is right-special.  Some of these suffixes, namely, the {\em constant} suffixes $a_2^{\ell_{k, 2}}, \ldots, a_1^{\ell_{k, 1}}$, are in $B$.  But many of these suffixes are not constant and therefore not in $B$.

For example, let $E_1(k)$ denote the set of suffixes of $a_1^{\ell_{k, 1}}w_{12}^{(k)}a_2^{\ell_{k, 2}}$ that are longer than $\ell_{k, 2}$ but no longer than $\ell_{k, 2} + \ell_{k, 1}$.  Then each $u \in E_1(k)$ is right-special and not in $B$.  

In general, for $1 \leq p \leq j$, let $E_p(k)$ denote the set of suffixes of $a_p^{\ell_{k, p}}w_{pq}^{(k)}a_q^{\ell_{k, q}}$ that are longer than $\ell_{k, q}$ but no longer than $\ell_{k,q} + \ell_{k, p}$, where $q = p+1$ if $2 \leq p < j$ and $q = 1$ if $p = j$.  Note that $\#E_p(k) = \ell_{k, p}$, and therefore \begin{equation} \label{Epkequation} \sum_{p = 1}^j \#E_p(k) = \ell_{k, 1} + \cdots + \ell_{k, j}.\end{equation}  Also, each $E_p(k) \cap B = \varnothing$, and, for $r \neq p$, $E_p(k) \cap E_r(k) = \varnothing$.

\begin{proposition} \label{countingprop}
For each $k \in \n$, \[c_X(\ell_{k, 1} + \cdots + \ell_{k, j}) \geq (j+1)(\ell_{k, 1} + \cdots + \ell_{k, j}) + (\ell_{k-1, 1} + \cdots + \ell_{k-1, j}).\]
\end{proposition}

\begin{proof}
To simplify notation, let $L_k = \ell_{k, 1} + \cdots + \ell_{k, j}$.  We proceed by induction, where, since $L_0 = 0$, the base case is the assertion that $c_X(L_1) \geq (j+1)L_1$.  This assertion follows from Lemma \ref{countinglemma} together with the following observations:\begin{itemize}
\item for $1 \leq \ell \leq L_1 - 1$, there are $j$ right-special words of length $\ell$ in $B$;
\item there are $L_1$ (distinct) right-special words in $\bigcup_{p = 1}^j E_p(1)$, none of which are in $B$.
\end{itemize}

Now assume $c_X(L_{k-1}) \geq (j+1)L_{k-1} + L_{k-2}$, and observe that, by Lemma \ref{countinglemma} together with Equation (\ref{Epkequation}), \begin{eqnarray*}
c_X(L_k) & \geq & c_X(L_{k-1}) + \sum_{\ell = L_{k-1}}^{L_k - 1} \#RS(\ell) \\ & \geq & \left[ (j+1)L_{k-1} + L_{k-2} \right] + \left[ j(L_k - L_{k-1}) + L_k \right] \\ & \geq & (j+1)L_k + L_{k-1}.
\end{eqnarray*}

\end{proof}

Since $(\ell_{k-1, 1} + \cdots + \ell_{k-1, j}) \rightarrow \infty$ as $k \rightarrow \infty$, it follows from Proposition \ref{countingprop} that $c_X(n) \greatersup (j+1)n$, which by our earlier discussion completes the proof of Theorem~\ref{maintheoremintro} in Cases (ii) and (iii).

\subsubsection{Sharpness of Theorem \ref{maintheoremintro} when $i = 0$} \label{sharpnessof1}  
Here we define a family of examples that demonstrates the sharpness of the bounds in Theorem \ref{maintheoremintro}.  We will consider the $i=0$ cases first, and then show how to modify the argument for $i > 0$.  Let $g: \n \rightarrow \n$ be a given nondecreasing unbounded function.

Define a sequence $\omega = \omega_1\omega_2\omega_3\cdots$ via the rule that $\omega_{2^mk + m} = m$ for $m \geq 1$ and $k \geq 0$, so that $\omega = 1213121412131215\cdots$.  Then, define a doubly-infinite sequence \[x = j^\infty . \, (1^{n_1^1}2^{n_1^2} \cdots j^{n_1^j})(1^{n_2^1}2^{n_2^2} \cdots  j^{n_2^j})(1^{n_1^1}2^{n_1^2} \cdots j^{n_1^j})(1^{n_3^1}2^{n_3^2} \cdots j^{n_3^j}) \cdots,\] for a doubly-indexed sequence $n_k^p$ satisfying \[n_1^1 << n_1^2 << \cdots << n_1^j << n_2^1 << n_2^2 \cdots << n_2^j << \cdots .\]
(The pattern of $n_k^p$ within $x$ is as follows: the superscript $p$ is always the same as the letter being repeated, and the subscript $k$ comes from 
$\omega$, in that it is $\omega_1 = 1$ for the first $j$ exponents, then $\omega_2 = 2$ for the next $j$, then $\omega_3 = 1$ for the next $j$, and so on.) Note that while $\omega$ is fixed, different sequences $(n_k^p)$ give rise to different sequences $x$.  In other words, $x$ represents a family of examples parametrized by $(n_k^p)$.

Regardless of choice of $(n_k^p)$, the transitive subshift $X = \overline{\mathcal O(x)}$ has $j$ minimal subsystems:  $M_1 = \{1^\infty\}$,\ldots, $M_j = \{j^\infty\}$.  Also, we can enumerate the right-special words in $X$ as follows.

All words in the previously-defined set $B$ are again right-special, and any right-special word not in $B$ must end with a word from the set \[ C = \{12^{n_k^2}, 23^{n_k^3}, \ldots, j1^{n_k^1} \; : \; k \in \n\}.\]  For $1 \leq p \leq j - 1$ and for any given $k$, the word \[p^{n_k^p}(p+1)^{n_k^{p+1}}\] is {\em maximally right-special} in the sense that: \begin{itemize}
	\item any suffix of $p^{n_k^p}(p+1)^{n_k^{p+1}}$ is right-special; and
	\item for any symbol $s$, the word $sp^{n_k^p}(p+1)^{n_k^{p+1}}$ fails to be right-special.
\end{itemize}  

Let $D = D(k)$ denote the set of words that are not in $B$ and are suffixes of one of these.  Each $n$ in an interval of the form $\ds \left(n_k^{p+1}, n_k^{p+1} + n_k^p \right]$ corresponds to the length of a word in $D$.  Moreover, we can ensure that these intervals are disjoint by requiring that 
\[n_k^3 > n_k^2 + n_k^1\; , \;\; n_k^4 > n_k^3 + n_k^2\; , \;\ldots \; , \;\; n_k^j > n_k^{j-1}+n_k^{j-2}. \]

Now consider right-special words ending in $j1^{n_k^1}$.  In $\omega$, the left-most occurrence of $k$ is $\omega_{2^k - 1}$; any occurrence of a symbol $m \geq k$ in $\omega$ is directly preceded by the word $\omega_1 \cdots \omega_{2^{k-1} - 1}$; and any occurrence of the word $\omega_1 \cdots \omega_{2^{k-1}-1}$ is directly followed by a symbol $m \geq k$.  Moreover, any occurrence of the word $k\omega_1 \cdots \omega_{2^{k-1} - 1}$ must be followed by a symbol $m > k$, and any occurrence of the word $m \omega_1 \cdots \omega_{2^{k-1} - 1}$ for $m > k$ must be followed by $k$.

It follows from these observations about $\omega$ that, in $x$, there is a unique word $u = u(k)$ of length \[L = L(k) = \sum_{i = 1}^j \left( n_{\omega_1}^i + n_{\omega_2}^i + \cdots + n_{\omega_{2^{k-1} - 1}}^i \right),\] namely, \[u := x_0 \cdots x_{L-1} = \underbrace{(1^{n_1^1} \cdots j^{n_1^j})(1^{n_2^1}\cdots j^{n_2^j}) \cdots (1^{n_1^1}\cdots j^{n_1^j})}_{2^k - 1 \mbox{ \tiny parenthetical blocks}} \] that directly precedes any occurrence of $1^{n_k^1}$.  It follows that $j^{n_k^j}u1^{n_k^1}$ is maximally right-special.  

Let $ F =  F(k)$ denote the set of words that are not in $B$ and are suffixes of $j^{n_k^j}u1^{n_k^1}$, and note that $ F \cap D = \varnothing$.  Then, for each $n \in (n_k^1, n_k^1 + n_k^j + L(k)]$, there is exactly one word of length $n$ in $F$, and that word is right-special.  

We have established the following.

\begin{proposition}\label{rtspec7}
The set of right-special words in $X$ is \[\bigcup_{n = 1}^\infty RS(n) = B \cup \bigcup_{k = 1}^\infty (  D(k) \cup   F(k)).\]  Moreover, \[\#RS(n) \leq \begin{cases} j & \mbox{ if $n_k^j + n_k^1 + L(k) < n \leq n_{k+1}^1$ for some $k$;} \\ j + 2 & \mbox{ if $n_k^{p} < n \leq n_k^{p} + n_k^{p-1}$ for some $k$ and $2 \leq p \leq j $;} \\  j+1&  \mbox{ otherwise. }\end{cases} \]
\end{proposition}

Set $n_0^j=0$. 
Since $g$ is nondecreasing and unbounded, we can choose the sequence $n_1^1<n_1^2<\cdots < n_1^j < n_2^1 < \cdots$ to grow fast enough so that for each $k \in \n$, and $p \in \{1,\ldots,j\}$, 
$g(n_k^p)$ is larger than the sum of all $n_{\ell}^q$ smaller than $n_k^p$. More specifically, 
choose $(n_k^p)$ so that for all $k \geq 1$ and $p \in \{1,\ldots ,j\}$
\[g(n_k^{p})  >  \sum_{\ell = 1}^{k-1} \sum_{q = 1}^{j} n_{\ell}^q + \sum_{q =1}^{p-1} n_{k}^q.\] 
Note that the right-hand side above provides an upper bound on the number of $n\in [1,n_k^{p+1})$ with $\#RS(n) = j+2$.
Lemma~\ref{countinglemma} then implies that for $n \in [n_k^p,n_k^{p+1})$,
\[c_X(n) \leq (j+1)n + g(n_k^p)\leq (j+1)n + g(n).\]
Similarly, if $n \in [n_{k-1}^j,n_k^1)$ for some $k \geq 1$, then 
\[c_X(n) \leq (j+1)n + g(n_{k-1}^j)\leq (j+1)n + g(n).\]
We've shown that $c_X(n) \leq (j+1)n + g(n)$ for all $n$. Since $g$ was arbitrary, this shows that the bound $c_X(n) \greatersup (j+1)n$ is sharp.

To see that the bound $c_X(n) \greaterinf jn$ is sharp, we consider the complexity along the subsequence $(n_k^1)$. If we choose $(n_k^p)$ to grow fast enough, then for all $k \geq 1$, 
\[g(n_{k+1}^1) > (j+2)(n_k^j+n_k^1+L(k)).\]
Then Lemma~\ref{countinglemma} and Proposition~\ref{rtspec7} imply that 
\begin{align*}
c_X(n_{k+1}^1) & \leq c_X(n_{k}^j+n_{k}^1+L(k)) + j(n_{k+1}^1-(n_{k}^j+n_{k}^1+L(k))) \\
& \leq (j+2)(n_k^j+n_k^1+L(k)) + j n_{k+1}^1 - j (n_{k}^j+n_{k}^1+L(k)) \\
& \leq g(n_{k+1}^1) + j n_{k+1}^1- j (n_{k}^j+n_{k}^1+L(k)).
\end{align*}
Since $c_X(n) \leq jn + g(n)$ along the sequence $n_k^1$ and $g$ was arbitrary, the bound $c_X(n) \greaterinf jn$ is sharp.

\subsubsection{Sharpness of Theorem \ref{maintheoremintro} when $i>0$}\label{sharpnessi>0} We now wish to show that the sharp examples constructed in Section \ref{sharpnessof1} can be extended to the $i > 0$ case. For this, first consider a system $X \in S(j, 0)$ constructed using the form from Section~\ref{sharpnessof1}, i.e., $X$ is the orbit closure of a sequence of the form
\[
x = j^{\infty}.(1^{n_1^1} 2^{n_1^2} \cdots j^{n_1^j}) (1^{n_2^1} 2^{n_2^2} \cdots j^{n_2^j}) (1^{n_1^1} 2^{n_1^2} \cdots j^{n_1^j}) (1^{n_3^1} 2^{n_3^2} \cdots j^{n_3^j}) \cdots .
\]

Now let $1 \leq i \leq j$, and let $S_{j-i+1}, \ldots, S_j$ be arbitrary Sturmian subshifts with alphabets disjoint from each other and from $\{1, \ldots, j-i\}$.  Our goal is to replace constant strings of symbols from $\{j-i+1, \ldots, j\}$ with sequences chosen from $S_{j-i+1}, \ldots, S_j$ to create $X' \in S(j,i)$ in such a way that the complexity is increased by exactly $in$.  We begin with an elementary observation about minimal subshifts, which applies in particular to $S_{j-i+1}, \ldots, S_j$.

\begin{lemma} \label{sturmian}
Given a minimal subshift $S$ and a word $w \in \mathcal L(S)$, there exist arbitrarily long words $v$ with the property that $wvw \in \mathcal L(S)$.
\end{lemma} 

\begin{proof}
By minimality, every point in $x\in S$ contains $w$ infinitely many times. Therefore, we can find two instances of $w$ in $x$ occurring at indices that are arbitrarily far apart. 
\end{proof}

  To stitch $S_{j-i+1}, \ldots, S_j$ into $X$, we need to impose a further assumption on the sequences $(n_m^i)_{m \in \mathbb{N}}$.  By Lemma \ref{sturmian}, we can recursively define $n_1^1$, $n_1^2, \ldots, n_1^j$, $n_2^1$, $n_2^2, \ldots, n_2^j$, $n_3^1, \ldots$ in such a way that, associated to each $p \in \{j-i+1, \ldots, j\}$ and $k \in \mathbb N$, there is a word $w_k^p \in \mathcal{L}_{n_k^p}(S_p)$, and every such $w_k^p$ is both a prefix and suffix of $w_{k+1}^p$. The proof of sharpness only required rapid growth of the sequence $(n_k^p)$, and Lemma \ref{sturmian} ensures that we may recursively choose $n_k^p$ with arbitrarily rapid growth such that the words $w_k^p$ have the desired conditions. Since each $w_k^j$ is a suffix of $w_{k+1}^j$, the sequence $w_k^j$ has a left-infinite limit (as $k \rightarrow \infty$), which we denote by $w_{\infty}^j$.

Now define
\[
x' = w_{\infty}^j.\big(1^{n_1^1} \cdots (j-i)^{n_1^{j-i}} w_1^{j-i+1} \cdots w_1^j\big) \big(1^{n_2^1} \cdots (j-i)^{n_2^{j-i}} w_2^{j-i+1} \cdots w_2^j\big) \cdots,
\]
the sequence obtained by replacing each $p^{n_k^p}$ in $x$ by $w_k^p$, and replacing $j^{\infty}$ by $w_{\infty}^j$. Then each $S_p$ for $p \in \{j-i+1, \ldots, j\}$ is an infinite minimal subsystem of $X' := \overline{\mathcal{O}(x')}$, and each $\{p^{\infty}\}$ for $p \in \{1, \ldots, j-i\}$ is a finite minimal subsystem of $X'$. It is not hard to check that $X'$ contains no other minimal subsystems, and so $X' \in S(j,i)$. It remains to show that $c_{X'}(n) = c_X(n) + in$. 

For this, consider the following $1$-block factor map $\phi$ applied to $X'$. Since the alphabets of the $S_p$ are disjoint, we may map any letter in the alphabet of $S_p$ to $p$, and leave other letters (for $p \in \{1, \ldots, j-1\}$) unchanged. The map $\phi$ induces a surjection from $\mathcal{L}_{n}(X')$ to $\mathcal{L}_n(X)$ for all $n$. We claim that every word in $\mathcal{L}(X)$ which is not constant has only a single $\phi$-preimage. To see this, consider $w \in \mathcal{L}(X)$ containing multiple letters. Without loss of generality, we can extend $w$ on the left and right so that $w$ contains some $a^{n_m^{a}}$ as a prefix and $b^{n_m^{b}}$ as a suffix; if the extension has only one preimage, then of course $w$ did as well. Then, by construction of $x'$, the only subword of $x'$ mapping to $w$ under $\phi$ is obtained by replacing every maximal subword of the form $p^{n_k^p}$ in $x$ by $w_k^p$ for $p \in \{j-i+1, \ldots, j\}$. (The only possible ambiguity comes from $a^{n_m^a}$ and $b^{n_m^b}$, but recall that $w_k^p$ is a prefix and suffix of $w_{k'}^p$ for all $k' > k$, and so even if $a^{n_m^a}$ and/or $b^{n_m^b}$ were portions of longer runs of $a$'s or $b$'s in $x$, the corresponding word in $x'$ still contains $w_m^a$ and/or $w_m^b$ at those locations if $a$ and/or $b$ are in $\{j-i+1,\ldots, j\}$.)

On the other hand, for $k \in \{j-i+1, \ldots, j\}$, any constant word of the form $k^n$ has every word in $\mathcal{L}_n(S_k)$ as a preimage, since $S_k \subset X'$, and all words in $\mathcal{L}_n(S_k)$ map to $k^n$ under $\phi$. Since $c_{S_k}(n) = n+1$ for all $n$, this means that all such words have $n+1$ preimages under $\phi$. Combining this yields $c_{X'}(n) = 1 \cdot (c_X(n) - i) + (n+1)i = c_X(n) + in$ for all $n$.

Now, the proof from Section  \ref{sharpnessof1} provides examples of $X \in S(j, 0)$ demonstrating sharpness of the bounds \[c_X(n) \greatersup (j+1)n\;\;\; \mbox{ and } \;\;\; c_X(n) \greaterinf jn\] in the $i = 0$ case. The procedure above yields, for any $0 \leq i \leq j$, $X' \in S(j,i)$ with $c_{X'}(n) = c_X(n) + in$, and so such $X'$ demonstrate the sharpness of the more general bounds (1) and (2) from Theorem~\ref{maintheoremintro}.

\subsection{Proof of Theorem \ref{Theorem2.5}}
We sketch the proof of Theorem \ref{Theorem2.5} here. Suppose $X$ is a transitive subshift $X$ with a recurrent transitive point $x$ such that $X$ properly contains an infinite minimal subshift $M$. 
Because $M \neq X$, there is an $r \geq 1$ such that $\mathcal{L}_r(M) \neq \mathcal{L}_r(X)$. 
Then define a factor map $\pi$ on $X$ such that 
\[\pi(z)_k = \begin{cases} 0  & \text{ if } z_k\cdots z_{k+r-1} \in  \mathcal{L}_r(M) \\
1 & \text{ otherwise. } 
\end{cases}\]
The image $\pi(x)$ is a recurrent transitive point for $\pi(X)$, and the subshift $\pi(X)$ contains a unique minimal subshift $\pi(M) = \{0^{\infty}\}$. Since $\pi(X) \neq \{0^{\infty}\}$, Theorem \ref{OrmesPavlovTheorem} gives $c_{\pi(X)}(n) \greatersup 1.5n$. 
Using an estimate as in Lemma \ref{factorlemma} with $i=1$ yields $c_X(n) \greatersup 2.5n$. 
Finally, Theorem~\ref{OrmesPavlovTheorem} also implies that there exist examples of such $\pi(X)$ demonstrating sharpness of 
$c_{\pi(X)}(n) \greatersup 1.5n$, and the reader may check that the examples constructed in \cite{OP} were similar to those
from Section~\ref{sharpnessof1}, relying only on rapid growth of an auxiliary sequence. Therefore, the same technique used 
in Section~\ref{sharpnessi>0} yields $X$ demonstrating sharpness of $c_X(n) \greatersup 2.5n$.


\section{General transitive systems} \label{nonrecurrent}

Our main theorem in the recurrent case, Theorem \ref{maintheoremintro}, provides bounds in terms of $i$ and $j$ for subshifts in the sets $S(j, i)$.  Our main theorem in the general transitive case, Theorem \ref{maintheoremnonrecurrentintro}, offers similar bounds in terms of sets that we will refer to as $T(j, i)$. Here $T(j,i) \supset S(j,i)$ is the set of all transitive subshifts with $j\geq 1$ minimal subsystems, exactly $i$ of which are infinite where $0 \leq i \leq j$. 

\subsection{Two or fewer subsystems}  We first tackle the cases where $j\leq 2$. When $j=1$, the results of Theorems \ref{SturmianTheorem} and \ref{MH} yield the minimal complexity sequence of $c_X(n)=n+1$ for $X$ not periodic. 

When $j=2$, the orbit closure of the sequence 
$$x=\ldots 0000.11111\ldots$$
produces a transitive system in $T(2,0)$ satisfying $c_X(n) = n+1$ for all $n$.

\begin{theorem} \label{nonrecurrentnotEPBbound}
	Let $X \in T(2,i)$ where $i >0$. Then
	\[\liminf_{n \rightarrow \infty}(c_X(n) - (i+1)n) > -\infty. \] Moreover, this bound is optimal in that the $-\infty$ cannot be replaced by any integer. 
\end{theorem}

\begin{proof}
Suppose $X$ contains two minimal subsystems $M_1$ and $M_2$. Then there is an $r>0$ such that $\mathcal{L}_r(M_1) \cap \mathcal{L}_r(M_2) = \varnothing$. Define a factor map  $\pi$ on $X$ such that 
\[\pi(z)_k = \begin{cases} 
i  & \text{ if } z_k\cdots z_{k+r-1} \in  \mathcal{L}_r(M_i), \text{ for } i=1,2 \\
0 & \text{ otherwise. } 
\end{cases}\]
Let $y=\pi(x)$. Then, for all $n \geq 1$, the words $1^n$ and $2^n$ occur in $y$. This means that $y$ is not periodic, so, by the Morse-Hedlund Theorem, $c_Y(n) \geq n+1$ for all $n$. Since $i > 0$, we may assume without loss of generality that $M_1$ is infinite, and so again by Morse-Hedlund, $c_{M_1}(n) \geq n+1$ for all $n$. 
Since $\pi$ is an $r$-block map, for $n \geq r$, the word $1^{n-r+1}$ has at least $c_{M_1}(n) \geq n+1$ 
$\pi$-preimages, so we obtain 
$$c_X(n) \geq c_{\pi(X)}(n-r+1)+n > 2n-r+1.$$

If $M_2$ is also infinite, then both $1^{n-r+1}$ and $2^{n-r+1}$ have 
at least $n+1$ preimages, and all of those preimages are distinct. 
Therefore, 
$$c_X(n) \geq c_{\pi(X)}(n-r+1)+2n > 3n-r+1.$$

We also claim that the $-\infty$ in Theorem~\ref{nonrecurrentnotEPBbound} cannot be replaced by any integer. We first treat the $i = 1$ case: choose any $N > 3$ and define a Sturmian subshift $Z \subset \{0,1\}^{\z}$ created by $\beta$ where $\frac{1}{N+1} < \beta < \frac{1}{N}$. Then $10^k1 \in \mathcal{L}(Z)$ if and only if $k = N-1$ or $N$. Consider a right-infinite word $s$ in $Z$, let $$x = 0^{\infty}.s,$$
and let $X = \overline{\mathcal{O}(x)}$.
Then, for $1 \leq n < N$, there is only one right-special word in $\mathcal{L}_n(X)$, namely, $0^n$. Therefore  $c_X(N) = N+1$. Sturmian systems 
contain exactly one right-special word of length $n$ for every $n \geq 1$. Therefore, when $n\geq N$, 
there are two right-special words in $X$: $0^n$ and a different one that is a subword of $s$. Lemma~\ref{countinglemma} then implies that $c_X(n) = 2n-N+1$ for $n \geq N$, so $\liminf c_X(n) - 2n = -N+1$. Since $N$ could be arbitrarily large, there is no uniform lower bound in the $i = 1$ case of Theorem~\ref{nonrecurrentnotEPBbound}.

For the $i = 2$ case, consider two Sturmian subshifts $Z_1, Z_2 \subset \{0,1\}^{\z}$ created by distinct $\beta_1,\beta_2 \in \left(\frac{1}{N+1} , \frac{1}{N}\right)$ for any $N>3$. Consider a left-infinite word $r \in Z_1$ ending in $1$ and a right-infinite word $s \in Z_2$ beginning with $1$. Let $$x = r.s,$$
and let $X = \overline{\mathcal{O}(x)}$.

There are three types of words in $x$: subwords of $r$, subwords of $s$, and words that contain the word $11$. For $n \geq 2$ there are exactly $n-1$ subwords of $x$ that contain $11$. For $1 \leq n \leq N$, $\mathcal{L}_n(Z_1)=\mathcal{L}_n(Z_2)$; both of these equal the set of words of length $n$ in $\{0,1\}^n$ that contain at most one 1. 
Therefore $c_X(n) = (n-1) + (n+1) = 2n$ for $1 \leq n \leq N$. 

For $n > N$, there are at most three right-special words: a single word $w_1$ in $\mathcal{L}_n(Z_1)$ that can be extended in two ways in $\mathcal{L}(Z_1)$, a single word $w_2$ in $\mathcal{L}_n(Z_2)$ that can be extended in two ways in $\mathcal{L}(Z_2)$, and the $n$-letter suffix $w_3$ of 
$r$. 
Therefore, by Lemma~\ref{countinglemma}, $c_X(n) = 3n - N$ for $n > N$, and so $\liminf c_X(n) - 3n = -N$, implying that there is no uniform lower bound in the $i = 2$ case of Theorem~\ref{nonrecurrentnotEPBbound}. \end{proof}

\subsection{Three or more subsystems}
We now proceed with the proof of Theorem \ref{maintheoremnonrecurrentintro}, which gives the bound $c_X(n) \greaterinf (j+i)n$ for $X \in T(j, i)$ with $j \geq 3$.  For such an $X$, let $x \in X$ be a transitive point.
Consider the same $r$-block factor map $\pi$ as constructed in Section \ref{pimap}. Then $\pi(X) \in T(j,0)$ for $j \geq 3$. Set $y=\pi(x)$ and note that $y$ is a transitive point for $\pi(X)$. We also note that $y$ cannot be eventually periodic in both directions, or else $\pi(X)$ would have at most two minimal subsystems (the periodic alpha-limit and omega-limit sets of $y$), a contradiction.

We then assume without loss of generality that $y$ is not eventually periodic to the right. (If $y$ were eventually periodic to the right, then below we would consider left-special, as opposed to right-special words, and arrive at the same conclusion.) 
	
For each $p \in \{1, \ldots, j\}$ and $n \in \n$, we claim $a_p^n$ is right-special. Indeed,  $a_p^{n+1}$ must occur in $y$ for all $p$ since $\{a_p^{\infty}\}$ is a minimal subsystem of $\pi(X)$. Since $y$ is not eventually periodic to the right, the word $a_p^nb$ must also occur in $y$. 

Therefore $c_{\pi(X)}(n) \geq jn$ for all $n \geq 1$. By Lemma~\ref{countinglemma}, we may establish 
$c_{\pi(X)}(n) \greaterinf jn$ by finding an infinite set of $n$ for which $\#RS_{\pi(X)}(n) \geq j+1$. Fix $p$ so that $\{a_p^{\infty}\}$ is in the omega-limit set $\omega(y)$. Then there exist infinitely many $n$ for which $ba_p^nb$ occurs in $y$. Recall that $ba_p^{n+1}$ occurs in $y$ for all $n\geq 1$, so $ba_p^n \in RS_{\pi(x)}(n+1)$, i.e. $\#RS_{\pi(x)}(n+1) \geq j+1$, for infinitely many $n$. 

We then have $c_{\pi(X)}(n) \greaterinf jn$. Finally, as in Lemma \ref{factorlemma}, for every $p$ with $\pi^{-1}\{a_p^{\infty}\}$ infinite, the number of 
$\pi$-preimages of $a_p^{n-r+1}$ is at least $n$, establishing \[c_{X}(n) \greaterinf (j+i)n.\]

\begin{remark} \rm \label{rem23}
The proof above also works if $X \in S(2,i)$ has a transitive point $x$ such that $\pi(x)$ is not both eventually periodic to the right and eventually periodic to the left. 
\end{remark}

\begin{theorem}
  Let $X \in T(j, i)$ where $j \geq 3$.  Then  the bound 
\[c_X(n)  \greatersup  (j+i)n\]
holds and is sharp. 
\end{theorem}
\begin{proof}
The bound itself follows immediately from Theorem \ref{maintheoremnonrecurrentintro}.  To see that the bound is sharp, let $g: \n \rightarrow \n$ be any nondecreasing unbounded function.  
Consider a point of the form \begin{equation} \label{sharpexTj0} x = 0^\infty .1^{n_1} 2^{n_2}  \cdots (j-1)^{n_{j-1}} 1^{n_j} 2^{n_{j+1}} \cdots  (j-1)^{n_{2j-2}}\cdots, \end{equation} where $n_1 << n_2 << n_3 << \cdots$, and 
set $X = \overline{\mathcal O(x)}$. 
  
For any $n \geq 1$, $0^n, 1^n, \cdots , (j-1)^n$ are all right-special in $X$. No word in $X$ containing $01$ is right-special, nor is any word that contains three distinct symbols.   The only other right-special words are of the form 
\[p^m (p+1)^{n_k} \text{ where } k \equiv (p+1) \mod (j-1), 1 \leq m \leq n_{k-1} \]
or 
\[(j-1)^m 1^{n_k} \text{ where } k \equiv 1 \mod (j-1), 1 \leq m \leq n_{k-1}. \]
In other words, we have $\#RS(n)=j+1$ only for $n \in (n_k,n_k+n_{k-1}]$. Therefore, if $n \in [n_k,n_{k+1} )$, the number of right-special words of length less than $n$ which are not of the form $a^n$ is at most
$\displaystyle \sum_{i =1}^{k-1} n_i$. 
If the sequence $(n_k)$ grows sufficiently fast, then $\displaystyle g(n_k)>\sum_{i =1}^{k-1} n_i$, and then Lemma~\ref{countinglemma} implies that
\[c_X(n)  < j n + g(n) \]
for all $n$. 

For the $T(j, i)$ bound where $j \geq 3$ and $i > 0$, consider the family of examples obtained by replacing the blocks $0^\infty$, $1^{n_1}$, $2^{n_2}$, ..., $(i-1)^{n_{i-1}}$, $1^{n_{j}}$, $2^{n_{j+1}}$, ..., $(i-1)^{n_{j+i-2}}$, ... with blocks from Sturmian sequences as in Section \ref{sharpnessi>0}. Then the estimate in Lemma \ref{factorlemma} gives the sharpness of the bound. 
\end{proof}


\section{Generic Measures} \label{measures}

We say that a point $x$ in a subshift $X$ is {\em generic} for a measure $\mu$ if the measures $\nu_n(x) := \frac{1}{n} \sum_{i = 0}^{n-1} \delta_{\sigma^i x}$ converge to $\mu$ in the weak topology, i.e., if 
\[\lim_{N \rightarrow \infty} \frac{1}{N} \sum_{i = 0}^{N-1} f(\sigma^i(x)) = \int f\, d\mu\] 
for all continuous $f: X \rightarrow \real$. The pointwise ergodic theorem implies that whenever $\mu$ is ergodic, $\mu$-almost every $x \in X$ is generic for $\mu$.
We say that a measure $\mu$ is a {\em generic measure} on $X$ if there exists $x \in X$ that is generic for $\mu$. All ergodic measures are clearly generic by the pointwise ergodic theorem, but generic measures are not necessarily ergodic; for instance, it is easily checked that 
\[
x = 0^{\infty}.011000111100000111111\ldots 
\]
is generic for $\mu = \frac{\delta_{0^{\infty}} + \delta_{1^{\infty}}}{2}$.

Our goal is to provide bounds on the number of generic measures a transitive subshift can support. One of the more general results in this vein is the following theorem of Cyr and Kra, which does not assume transitivity of the subshift, but also does not control for the number of atomic measures supported on periodic subshifts.

\begin{theorem}[\cite{CK}] \label{CKsup}
	Suppose $X$ is a subshift, and there exists $k \geq 3$ such that 
	\[\limsup_{n \to \infty} \frac{c_X(n)}{n} < k. \]
	If $X$ has a generic measure $\mu$ and there is a generic point $z \in X$ for $\mu$ such that the subshift $Z=\overline{\mathcal{O}(z)}$ is not uniquely ergodic, then $X$ has at most $k-2$ distinct, nonatomic, generic measures. 
\end{theorem}

Combining the Cyr-Kra result above with others in this paper we obtain the same conclusion under a different hypothesis. 
\begin{theorem}
	Suppose $X$ is a transitive subshift which is the orbit closure of a recurrent point and there exists $k \geq 3$ such that 
	\[\limsup_{n \to \infty} \frac{c_X(n)}{n} < k. \]
	Then $X$ has at most $k-2$ distinct, nonatomic, generic measures. 
\end{theorem}

\begin{proof}
	Let $X$ be a transitive subshift with a recurrent transitive point. Fix a generic measure $\mu$ for $X$, and let $z$ be a generic point for $\mu$. If $\overline{\mathcal{O}(z)}$ is not uniquely ergodic, then we are done by Theorem \ref{CKsup}. Thus we may assume that $Z=\overline{\mathcal{O}(z)}$ is uniquely ergodic. Since $Z$ is uniquely ergodic, it follows that every point $x \in Z$ is generic for $\mu$ (see, for instance, \cite{W}). 
	
	Now assume for a contradiction that $X$ has $(k-1) \geq 2$ generic measures $\mu_1,\ldots, \mu_{k-1}$, with respective generic points $z_1, \ldots, z_{k-1}$. Then by the preceding paragraph, we obtain $(k-1)$ disjoint subsystems $Z_1 \ldots , Z_{k-1}$ such that for all $x \in Z_i$, $x$ is generic for $\mu_i$. Each $Z_i$ contains a distinct minimal subsystem. Therefore, by Theorem \ref{maintheoremintro}, 
	\[ \limsup_{n \to \infty} (c_X(n) -kn) = \infty, \]
	a contradiction. 
\end{proof}

In general, there does not appear to be a simple way to combine our results with those of \cite{CK} in order to use an upper bound on $\limsup_{n \to \infty} \frac{c_X(n)}{n}$ to bound the number of generic (not necessarily nonatomic) measures in the transitive case. However, with some effort, we are able to show that in the case where $k=3$, i.e., when $X$ is transitive with a recurrent transitive point and
	\[\limsup_{n \to \infty} \frac{c_X(n)}{n} < 3, \]
then $X$ is uniquely ergodic; this is our Theorem \ref{measuresupintro}. 

\subsection{Proof of Theorem \ref{measuresupintro}}  If $X$ is minimal, then this is Theorem 1.5 from \cite{B}.  If $X$ is not minimal, then it cannot be the case that $X$ contains two or more minimal subsystems, since our Theorem \ref{maintheoremintro} would imply that $c_X(n) \greatersup 3n$, which 
would contradict $\limsup_{n \rightarrow \infty}(\frac{c_X(n)}{n}) < 3$.

So $X$ contains a unique minimal subsystem $M_1$. By Theorem 1.5 from \cite{B}, there is a unique (ergodic) measure $\mu$ supported on $M_1$. Consider the factor map $\pi$ (as defined in sub-section \ref{pimap}); then $\mu$ pushes forward under $\pi$ to $\delta_{a_1^\infty}$ in $\pi(X)$, and it is the only measure which does so. So, if we are able to prove that $\pi(X)$ is uniquely ergodic, then its unique measure is $\delta_{a_1^\infty}$, which implies that $\mu$ is the unique measure on $X$.

Since the hypotheses of Theorem~\ref{measuresupintro} are preserved under application of a factor map, we can assume without loss of generality that $X$ has a unique minimal subsystem $\{a_1^{\infty}\}$. We can further reduce (by applying a $1$-block factor map sending $a_1$ to $0$ and all other letters to $1$) to the case where $X \subseteq \{0, 1\}^\z$ and that $X$ has unique minimal subsystem $\{0^{\infty}\}$.  
Toward a contradiction, suppose that such an $X$ has an ergodic $\mu \neq \delta_{0^\infty}$.  

\begin{lemma} \label{genericstructure}
Let $(n_k) \subseteq \n$ be a strictly increasing sequence and $x$ a generic point for $\mu$.  Then for all sufficiently large $k$, $x_{[0, \infty)}$ has the form \[x_{[0, \infty)} = .w_0^{(k)} \; 0^{\geq n_k} \; w_1^{(k)} \; 0^{\geq n_k} \; w_2^{(k)} \; \cdots,  \] where every $w_i^{(k)}$ begins and ends with $1$ and does not contain $0^{n_k}$.
Moreover, $\frac{|w_0^{(k)}|}{n_k} \rightarrow \infty$.
\end{lemma}

\begin{proof}

We first note that since $x$ is generic for $\mu \neq \delta_{0^\infty}$, $x_{[0,\infty)}$ contains infinitely many $1$s. Moreover, since $0^{\infty}$ is the only minimal subsystem of $X$, the omega-limit set of $x$ must contain $0^{\infty}$, i.e. $x_{[0,\infty)}$ contains $0^n$ for arbitrarily large $n$. Therefore, $x_{[0,\infty)}$ has the claimed form, and it remains only to show that $\frac{|w_0^{(k)}|}{n_k} \rightarrow \infty$.

By genericity, \begin{eqnarray*}
\mu([0]) & = & \int \chi_{[0]} \, d\mu \\ 
& = & \lim_{N \rightarrow \infty} \frac{1}{N} \sum_{i = 0}^{N-1} \chi_{[0]}(\sigma^i(x)) \\
& = &\lim_{N \rightarrow \infty} \frac{\#\mbox{ zeros in }x_0 \cdots x_{N-1}}{N}. 
\end{eqnarray*}

Observe that $\mu([0]) < 1$ since we are assuming $\mu \neq \delta_{0^\infty}$.  It follows that $|w_0^{(k)}| \rightarrow \infty$.  Therefore, if we let $r_k = |w_0^{(k)}|$ and $z_k$ denote the number of zeros in $w_0^{(k)}$, then \[\lim_{k \rightarrow \infty} \frac{z_k}{r_k} = \mu([0]) \; \mbox{ and } \lim_{k \rightarrow \infty} \frac{z_k + n_k}{r_k + n_k} = \mu([0]).\]  Through some algebraic manipulation, \[
0  =  \lim_{k \rightarrow \infty} \left( \frac{z_k}{r_k} - \frac{z_k + n_k}{r_k + n_k} \right)  =  \lim_{k \rightarrow \infty} \left( \left(\frac{n_k}{n_k + r_k}\right) \left( \frac{z_k - r_k}{r_k} \right) \right).\]
 Now, \[\lim_{k \rightarrow \infty} \left( \frac{z_k - r_k}{r_k} \right) = \lim_{k \rightarrow \infty}\left( \frac{z_k}{r_k} - 1 \right) = \mu([0]) - 1 \neq 0,\] so $\ds \lim_{k \rightarrow \infty}\left( \frac{n_k}{n_k + r_k}\right) = 0$, which implies that $\ds \frac{r_k}{n_k} \rightarrow \infty$.
\end{proof}

\begin{lemma} \label{rtspecial}
If $x$ is a sequence which is not eventually periodic to the right, $n \in \mathbb{N}$, and $w$ is a subword of $x$ with length at least $c_X(n) + n$, then $w$ contains an $n$-letter right-special word.
\end{lemma}

\begin{proof} 
Assume that $x$ and $w$ are as in the lemma. By assumption, $w$ contains more than $c_X(n)$ subwords of length $n$, and therefore one is repeated, call it $u$. 
If no $n$-letter subword of $w$ is right-special, then for every $n$-letter subword of $w$, there is only one choice of a letter which may follow it in $x$. However, this would mean that the portion of $x$ between the two occurrences of $u$ would have to repeat indefinitely to the right, a contradiction to the assumption that $x$ is not eventually periodic to the right. Therefore, $w$ contains an $n$-letter right-special word.
\end{proof}

\begin{lemma} \label{noones}
There does not exist a strictly increasing sequence $(n_k)$ such that $\#RS(n_k) = 1$ for all $k$.\end{lemma}

\begin{proof} 
If such a sequence $(n_k$) exists, then for every $k$, the only right-special word of length $n_k$ is $0^{n_k}$.  Let $x$ be generic for $\mu$. We note that $x$ cannot be right-asymptotic to a periodic point: if that periodic point were $0^\infty$, then by genericity $\mu = \delta_{0^\infty}$; and if it were not $0^\infty$, then $X$ would contain a minimal subsystem other than $\{0^\infty\}$.  

Consider $w_0^{(k)}$ as defined in Lemma \ref{genericstructure}. 
We note that $w_0^{(k)}$ cannot contain a right-special word of length $n_k$, since by definition it would not be $0^{n_k}$, a contradiction. 
Therefore, by Lemma~\ref{rtspecial}, $|w_0^{(k)}| < c_X(n_k) + n_k$.  

Next we observe that there exists $K \in \n$ such that $c_X(n) < Kn$ for all $n$.  To see this, note that, since $\limsup_{n \to \infty} (\frac{c_X(n)}{n}) < 3$, there exists $N \in \n$ such that $c_X(n) < 3n$ for all $n \geq N$.  And for $1 \leq n \leq N-1$, we can simply let $K(n) \in \n$ be any value such that $K(n) > c_X(n)$.  Then just define \[K := \max\{3, K(1), \ldots, K(N-1)\}.\]  

Then $|w_0^{(k)}| < c_X(n_k) + n_k < (K+1)n_k$ for all $k$, which contradicts the conclusion from Lemma \ref{genericstructure} that $\frac{|w_0^{(k)}|}{n_k} \rightarrow \infty$.

\end{proof}

Following \cite{B}, we define a strictly increasing sequence $(n_k)$ to be {\em logarithmically syndetic} if there exists $M$ such that $\ds \frac{n_{k+1}}{n_k} < M$ for all $k$.

\begin{lemma} \label{loglemma}
There is a logarithmically syndetic sequence $(n_k)$ such that $\#RS(n_k) \leq 2$ for all $k$.  Moreover, if $RS(n_k) = \{0^{n_k}, b_k\}$ for some $k$ and $b_k \neq 0^{n_k}$, then $b_k$ cannot be followed by three or more symbols.
\end{lemma}

\begin{proof}
Let  $\varepsilon = 3 - \limsup_{n \to \infty} (\frac{c_X(n)}{n})$.  Then there exists $N$ such that $\frac{c_X(n)}{n} < 3 - \varepsilon/2$ for all $n \geq N$.  We claim that any interval of the form $ \left[ i, \left(\frac{6}{\varepsilon} \right)\cdot i \right]$ where $i \geq N$ must contain a value of $n$ such that $c(n+1) - c(n) \leq 2$.  The statement of Lemma \ref{loglemma} then follows immediately by Lemma~\ref{countinglemma}.  

To prove the claim, let $i \geq N$ and suppose $c_X(n+1) - c_X(n) \geq 3$ for all $n \in \left[ i, j \right]$, where $j = \left(\frac{6}{\varepsilon} \right)\cdot i$.  Then \begin{eqnarray*}
c_X(j) & \geq & c_X(i) + 3(j - i)\\
& > & 3(j - i) \\
& = & j( 3 - \varepsilon/2).
\end{eqnarray*} But $\frac{c_X(j)}{j} < 3 - \varepsilon/2$ since $j \geq N$.
\end{proof}

Let $(n_k)$ be a fixed logarithmically syndetic sequence as given by Lemma \ref{loglemma}, and let $M>1$ be a constant such that $\frac{n_{k+1}}{n_k} < M$.  By Lemma \ref{noones}, by deleting a finite number of terms in $(n_k)$ if necessary, we can assume that $\#RS(n_k) = 2$ for all $k$.  Let $b_k$ denote the only right-special word of length $n_k$ other than $0^{n_k}$. By passing to a subsequence and increasing $M$ to $10M$ if necessary, we can assume that \begin{equation} \label{second8inequality} 10 < \frac{n_{k+1}}{n_k} < M \;  \;\;\mbox{ for all $k$.} \end{equation} 

Since $\frac{|w_0^{(k)}|}{n_k} \rightarrow \infty$ and $\limsup_{k \to \infty}  \frac{c_X(n_k)}{n_k} \leq \limsup_{n \to \infty} \frac{c_X(n)}{n} < 3$, we see that 
$|w_0^{(k)}| > c_X(n_k) + n_k$ for sufficiently large $k$. By Lemma~\ref{rtspecial}, $w_0^{(k)}$ contains a right-special word of length $n_k$, which cannot be $0^{n_k}$, so $w_0^{(k)}$ contains at least one occurrence of $b_k$.

Define $a_k$ to be the shortest (possibly empty) word such that \begin{equation} \label{akdefinition} w_0^{(k)} = a_kb_kr_k\end{equation} for some word $r_k$, and define $d_k$ to be the shortest word such that \begin{equation} \label{dkdefinition} w_0^{(k)} = \ell_kb_kd_k\end{equation} for some word $\ell_k$.   Therefore we can write \[w_0^{(k)} = a_kb_ku_kd_k\] for some (possibly empty) word $u_k$. Observe that $|a_k|, |d_k| < c_X(n_k) + n_k < 4n_k$ for all sufficiently large $k$, since otherwise, by Lemma~\ref{rtspecial}, $a_k$ or $d_k$ would contain $b_k$, contradicting minimality of length in their definitions.  By definition, $b_k$ is a suffix of $a_kb_ku_k$, but since $\frac{|w_0^{(k)}|}{n_k} \rightarrow \infty$, for sufficiently large $k$ we can assume that \begin{eqnarray*}
|w_0^{(k)}| & > & 13n_k \\
& = & 4n_k + n_k + 4n_k + 4n_k \\
& > & |a_k| + |b_k| + |d_k| + c_X(n_k) + n_k,
\end{eqnarray*}
which implies that $|u_k| > c_X(n_k)  + n_k$. Therefore, by Lemma~\ref{rtspecial}, $u_k$ contains at least one copy of $b_k$.

\begin{remark} \label{bkcopies} \rm
Let $p> 0$ be an arbitrary positive integer. As above, for sufficiently large $k$ we have  
\begin{eqnarray*}
	|w_0^{(k)}| & > & (13+4p)n_k \\
	& > & |a_k| + |b_k| + |d_k| + p(c_X(n_k) + n_k).
\end{eqnarray*}
Applying Lemma~\ref{rtspecial}, we see that for sufficiently large $k$, $u_k$ contains at least $p$ copies of $b_k$.
\end{remark} 

Observe that $b_k$ is a suffix of $a_kb_ku_k$, and $b_k$ is the only word in $RS(n_k)$ that appears in $w_0^{(k)}$.  Also, $a_k$ and $d_k$ are the {\em shortest} words that satisfy (\ref{akdefinition}) and (\ref{dkdefinition}), and (by Lemma \ref{loglemma}), $b_k$ can only be followed by two distinct symbols.  Therefore $w_0^{(k)}$ must have the form \begin{equation} \label{w0form} w_0^{(k)} = a_kb_k(e_kb_k)^{m_k}d_k,\end{equation} where $e_k$ is the shortest (possibly empty) word such that $w_0^{(k)} = a_kb_ke_kb_kq_k$ for some word $q_k$.  

In the form (\ref{w0form}), observe (by Lemma~\ref{rtspecial}) that \begin{equation} \label{ekinequality} |e_k| < c_X(n_k) + n_k < 4n_k \end{equation} for all sufficiently large $k$ since $e_k$ does not contain any right-special word of length $n_k$.   Also, $m_k \rightarrow \infty$ by Remark \ref{bkcopies}.  Therefore, by deleting a finite number of terms from the beginning of $(n_k)$ if necessary, we can assume that \begin{equation} \label{first8inequality} m_k > M^2 + 1 \;\;\; \mbox{ for all $k\geq 1$.}\end{equation} 

Observe that the first symbols of $e_k$ and $d_k$ must be different: if they were the same, then, since $b_k$ is the only word in $RS(n_k)$ that appears in $w_0^{(k)}$, we would have $e_k = d_k$ and $x_{[0, \infty)} = a_kb_k(e_kb_k)^\infty$.  As noted, for example, in the proof of Lemma \ref{noones}, this cannot happen because $x$ cannot be eventually periodic to the right.  

Now observe that all suffixes of $(e_kb_k)^{m_k-1}$ of length at least $n_k$ are right-special: this is due to the fact that any such suffix can be followed by either $e_k$ or $d_k$, and, as we just noted, the first symbols of $e_k$ and $d_k$ must be different.   By construction, none of these suffixes have the form $0^n$.  Therefore $\#RS(n) \geq 2$ for each $n \in (n_k, (m_k - 1) \cdot |e_kb_k|]$.

\begin{lemma} \label{firstbklemma}
For all $k$, both $b_{k+1}$ and $b_{k+2}$ are suffixes of $(e_kb_k)^{m_k-1}$.
\end{lemma}

\begin{proof}
Using (\ref{second8inequality}) and (\ref{first8inequality}), observe that \begin{eqnarray*}
|b_{k+1}| & = & n_{k+1} \\ & < & M \cdot n_k \\ & < & (m_k-1) \cdot n_k \\ & \leq & (m_k - 1) \cdot |e_kb_k|.
\end{eqnarray*}  Therefore $|b_{k+1}| \in (n_k, (m_k-1) \cdot |e_kb_k|)$.  And $b_{k+1}$ is the only right-special word of length $n_{k+1}$ other than $0^{n_{k+1}}$, which implies that $b_{k+1}$ is a suffix of $(e_kb_k)^{m_k-1}$.  

Next, observe that by (\ref{second8inequality}) and (\ref{first8inequality}), \begin{eqnarray*}
|b_{k+2}| & = & n_{k+2} \\ & < & M \cdot n_{k+1} \\ & < & M^2 \cdot n_k \\ & < & (m_k-1) \cdot n_k \\ & \leq & (m_k - 1) \cdot |e_kb_k|.
\end{eqnarray*}  Therefore $b_{k+2}$ is also a suffix of $(e_kb_k)^{m_k - 1}$.
\end{proof}

\begin{lemma} \label{secondbklemma}
For all $k$, $e_kb_ke_kb_k$ is a suffix of $b_{k+1}$.
\end{lemma}

\begin{proof}
By Lemma \ref{firstbklemma}, $b_{k+1}$ is a suffix of $(e_kb_k)^{m_k-1}$, so it is enough to show that $|b_{k+1}| \geq |e_kb_ke_kb_k|$.  To see this, observe that, using (\ref{second8inequality}) and (\ref{ekinequality}), \begin{eqnarray*}
|b_{k+1}| & = & n_{k+1} \\ & > & 10n_k \\ & = & 4n_k + n_k + 4n_k + n_k \\ & > & |e_k| + |b_k| + |e_k| + |b_k| \\ & = & |e_kb_ke_kb_k|. 
\end{eqnarray*}
\end{proof}

\begin{proposition} \label{ekbkprop}
For all $k$, $e_{k+1}b_{k+1}e_{k+1}b_{k+1}$ is a suffix of $(e_kb_k)^{m_k - 1}$.
\end{proposition}

\begin{proof}
If we apply Lemma \ref{secondbklemma} to $k+1$ instead of $k$, we see that $e_{k+1}b_{k+1}e_{k+1}b_{k+1}$ is a suffix of $b_{k+2}$, which, by Lemma \ref{firstbklemma}, is a suffix of $(e_kb_k)^{m_k-1}$.
\end{proof}

Finally, to arrive at a contradiction, we will use the words $e_kb_ke_kb_k$ to construct a minimal subsystem of $X$ other than $\{0^\infty\}$.  To do this, first define $L_k$ to be the length of the longest block of zeros occurring in $e_kb_ke_kb_k$.  Then $L_k$ is also the length of the longest block of zeros in $(e_kb_k)^{m_k - 1}$ (assuming $m_k \geq 3$, which is trivial since $m_k \rightarrow \infty$).  

\begin{proposition} \label{constantsequence}
$(L_k)$ is a constant sequence.
\end{proposition}  

\begin{proof}
By Proposition \ref{ekbkprop}, $e_{k+1}b_{k+1}e_{k+1}b_{k+1}$ is a suffix of $(e_kb_k)^{m_k - 1}$, so $L_{k+1} \leq L_k$.  On the other hand, by Lemma \ref{secondbklemma}, $e_kb_ke_kb_k$ is a suffix of $b_{k+1}$, so $L_k \leq L_{k+1}$.
\end{proof}

For each $k$, define $y^{(k)}$ to be any point in $X$ of the form \[y^{(k)} = \ell_k \, e_k \, b_k \, . \, e_k \, b_k \, r_k,\] where $\ell_k$ and $r_k$ are left- and right-infinite sequences.  By compactness, a subsequence of $(y^{(k)})$ converges to a point $y \in X$.  But by construction (and applying Proposition \ref{constantsequence}), $\overline{\mathcal O(y)}$ is a subsystem of $X$ that is disjoint from $\{0^\infty\}$.  Therefore $\overline{\mathcal O(y)}$ contains a minimal subsystem other than $\{0^\infty\}$, a contradiction, completing the proof of Theorem~\ref{measuresupintro}.

\begin{flushright}
	$\square$
\end{flushright}

\begin{remark} \label{rem23again} \rm
The assumption of recurrence was only used above to establish that $X$ has a single minimal subsystem. Our results in this paper show that the only transitive $X$ with more than one minimal subsystem and $\limsup_{n \to \infty} \left( \frac{c_X(n)}{n} \right)< 3$ are of the following types. 
\begin{itemize}
	\item $X$ has two periodic minimal subsystems (e.g. the orbit closure of $x =\ldots 000.1111 \ldots$ or $x= 1^{\infty}.0^{n_1} 1 0^{n_2} 1 0^{n_3} 1 \ldots$)
	\item $\pi(X)$ is eventually periodic in both directions and $X$ has one infinite subsystem (e.g., $x=\ldots 000.s \ldots $ where $s$ is a one-sided sequence from a Sturmian system).
\end{itemize}
\end{remark}

We now conclude with a proof of Theorem \ref{measureinfintro}, which establishes an upper bound of $g-1$ on the number of generic measures when $X = \overline{\mathcal{O}(x)}$ for some $x$ that is not eventually periodic in both directions and 
$$\liminf(c_X(n) - gn) = -\infty.$$ 

\subsection{Proof of Theorem \ref{measureinfintro}} \label{proofofmeasureinf}

Assume that $X$ is as above and $x$ is not eventually periodic to the right.  Since $\liminf(c_X(n) - gn) = -\infty$, Lemma~\ref{countinglemma} implies that the number of $n$-letter right-special words is strictly less than $g$ for infinitely many $n$. Therefore, there exists $C \leq g-1$ and a strictly increasing sequence $(n_k)$ such that there are exactly $C$ right-special words of length $n_k$; call them $b^{(i)}_k$ for $1 \leq i \leq C$. We may further assume that \[c_X(n_k) < 2g n_k\] for every $k$ by considering values of $n$ where $c_X(n) - gn$ is smaller than all previous values $c_X(i) - gi$ for $i < n$; see the discussion following Theorem 2.2 in \cite{B} for details. Therefore, by Lemma \ref{rtspecial}, every word of length 
$(2g + 1)n_k$ contains at least one word $b^{(i)}_k$.

For each $k$ and $i$, choose a point $x^{(i)}_k$ with the word $b^{(i)}_k$ appearing in coordinates $0$ through $n_k-1$, and define the measure $\nu^{(i)}_k := \nu_{n_k}^{x^{(i)}_k}$ as discussed at the start of Section \ref{measures}, i.e.,
\[\nu^{(i)}_k = \frac{1}{n_k} \sum_{j = 0}^{n_k - 1} \delta_{\sigma^j x^{(i)}_k}.\] By compactness, we can pass to a subsequence and assume without loss of generality that for all $i$, $\big(\nu^{(i)}_k\big)_k$ converges to a limit $\nu^{(i)}$. We claim that every generic measure is some $\nu^{(i)}$. For this, assume that $\mu$ is an arbitrary generic measure, and let $x \in X$ be generic for $\mu$.


For each $k$, the word $x_0 \cdots x_{(2g+1)n_k - 1}$ contains some $b^{(i_k)}_k$ by Lemma~\ref{rtspecial}, and by again passing to a subsequence, we can assume that $i_k$ is always equal to some fixed $i$. For each $k$, define $m_k \leq 2gn_k$ so that $x_{m_k} \ldots x_{m_k + n_k - 1} = b^{(i)}_k$.

Since $x$ is generic for $\mu$, both $\nu_{m_k}^x$ and $\nu^x_{m_k + n_k - 1}$ converge to $\mu$ in the weak topology. This topology is induced by the metric
\[
d(\mu, \nu) := \sum_{n \in \mathbb{N}} 2^{-n} |\mu([w_n]) - \nu([w_n])|,
\]
where $\{w_n\}$ is an arbitrary enumeration of the set of finite words on $\{0,1\}$. Then for any $\epsilon > 0$, we may choose $K$ so that for all $k > K$, 
$d(\nu_{m_k}^x, \mu), d(\nu_{m_k + n_k}^x, \mu) < \epsilon$. This implies that \[d(m_k \nu_{m_k}^x, m_k \mu) < m_k \epsilon \;\; \textrm{ and } \;\; d((m_k + n_k) \nu_{m_k + n_k}^x, (m_k + n_k) \mu) < (m_k + n_k) \epsilon,\] which together imply that \[d((m_k + n_k) \nu_{m_k + n_k}^x - m_k \nu_{m_k}^x, n_k \mu) < (2m_k + n_k) \epsilon.\]  Therefore \[d(\nu_{n_k}^{\sigma^{m_k} x}, \mu) < \frac{2m_k + n_k}{n_k} \epsilon \leq (4g + 1) \epsilon.\]

Since $\epsilon > 0$ was arbitrary, $\nu_{n_k}^{\sigma^{m_k} x} \rightarrow \mu$. Recall that $\sigma^{m_k} x$ begins with $b^{(i)}_k$, and so $\nu_{n_k}^{\sigma^{m_k} x} \rightarrow \nu^{(i)}$ by definition of $\nu^{(i)}$. Therefore, $\mu = \nu^{(i)}$, and since $\mu$ was an arbitrary generic measure, every generic measure is one of the $\nu^{(i)}$. Since there are $C \leq g-1$ such measures, the proof is complete.

\begin{flushright}
$\square$
\end{flushright}


\end{document}